\documentclass[10pt,reqno,letterpaper]{amsart}

\usepackage[ams]{optional} 

\usepackage{amsfonts,amssymb,amsmath}
\opt{ams}{\usepackage{amsthm}}

\usepackage{bbm, float, comment, setspace}
\usepackage{xstring}

\restylefloat{figure}

\opt{ams}{
\usepackage[text={160mm,220mm},centering]{geometry}

}

\usepackage{tikz, tikz-cd}
\usetikzlibrary{matrix,arrows}
\tikzset{commutative diagrams/diagrams={baseline=-3pt}}

\usepackage{svn-multi}
\svnid{$Id: surface-deformations.tex 1043 2013-01-15 18:22:35Z willdonovan $}
\svnidlong{$HeadURL$} {$LastChangedDate$} {$LastChangedRevision$} {$LastChangedBy$}

\DeclareFontFamily{OT1}{rsfs}{}
\DeclareFontShape{OT1}{rsfs}{n}{it}{<-> rsfs10}{}
\DeclareMathAlphabet{\curly}{OT1}{rsfs}{n}{it}

\newcommand{\xyExactTriangle}[3]{\ensuremath{\xymatrix@C=.6em{ & #3 \ar@{-->}[dl] & \\ #1 \ar[rr] & & #2 \ar[ul] }}}

\newcommand{\exactTriangleWithMaps}[6]{\ensuremath{\xymatrix@C=.6em{ & #5 \ar@{-->}_{#6}[dl] & \\ #1 \ar^{#2}[rr] & & #3 \ar_{#4}[ul] }}}
\newcommand{\xyFlatExactTriangle}[3]{\ensuremath{\xymatrix{#1\ar[r] & #2 \ar[r] & #3 \ar@{-->}[r] &}}}

\newcommand{\flatExactTriangleWithMaps}[6]{\ensuremath{\xymatrix{#1\ar^{#2}[r] & #3 \ar^{#4}[r] & #5 \ar@{-->}^{#6}[r] &}}}

\newcommand{\variableLengthPostnikov}[9]{\ensuremath{\xymatrix@!R@R=-3em@C=-3em{ & #1 \ar[rr] \ar[dr] & & #2 \ar[dr]  & \ldots & \ldots & \ldots & #3 \ar[rr] \ar[dr] & & #4 \ar[dr] & \\ #5 \ar[ur] & & #6 \ar@{-->}[ll] \ar[ur] & & #7 \ar@{-->}[ll]& \ldots  & #8 \ar[ur] & & #9 \ar[ur] \ar@{-->}[ll] & & 0 \ar@{-->}[ll]}}}

\newcommand{\leftEndPostnikov}[9]{\ensuremath{\xymatrix@!R@R=-2em@C=-2em{ & #1 \ar[rr] \ar[dr] & & #2 \ar[dr]  & \ldots & \\ #5 \ar[ur] & & #6 \ar@{-->}[ll] \ar[ur] & & #7 \ar@{-->}[ll]& \ldots }}}

\newcommand{\rightEndPostnikov}[9]{\ensuremath{\xymatrix@!R@R=1em@C=1em{ & \ldots & #3 \ar[rr] \ar[dr] & & #4 \ar[dr] & \\ \ldots & #8 \ar[ur] & & #9 \ar[ur] \ar@{-->}[ll] & & 0 \ar@{-->}[ll]}}}

\newcommand{\verticalPostnikov}[9]{\ensuremath{\opt{ams}{\xymatrix@!R=0.05em}\opt{lms}{\xymatrix@R=0.4em@C=0em}{& #5 \ar[ld] \\ #1 \ar[dd] \ar[rd] & \\ & #6 \ar@{-->}[uu] \ar[ld] \\ #2 \ar[rd] & \\ \vdots & #7 \ar@{-->}[uu] \\ \vdots & \vdots \\ \vdots & #8 \ar[ld] \\ #3 \ar[dd] \ar[rd] & \\& #9 \ar@{-->}[uu] \ar[ld] \\ #4 \ar[rd] & \\& 0 \ar@{-->}[uu] \\}}}

\newcommand{\quotes}[1]{\textquoteleft{#1}'}

\newcommand{\cone}[1]{\ensuremath{\operatorname{Cone}\left({#1}\right)}}
\newcommand{\setconds}[2]{\ensuremath{\{\begin{array}{c|c} #1 & #2 \end{array}}\}}

\newcommand\C{\mathbb C}

\renewcommand\k{\mathfrak k}

\newcommand\cO{\mathcal O}
\newcommand\PP{\mathbb P}

\newcommand\Z{\mathbb Z}

\newcommand\To{\longrightarrow}

\newcommand\onto{\twoheadrightarrow}
\newcommand\into{\hookrightarrow}

\newcommand\iso{\simeq}

\newcommand\xyhook{\ar@{^{(}->}}
\newcommand\xyhookup{\ar@{^{(}->}}
\newcommand\xyhookdown{\ar@{_{(}->}}

\newcommand\xybend{\ar@/^/}
\newcommand\xydoublebend{\ar@/^2pc/}
\newcommand\xybendup{\ar@/^/}
\newcommand\xybenddown{\ar@/_/}

\newcommand\xyequals{\ar@{=}}
\newcommand\xysubset{\ar@{}|-*{\subset}}

\newcommand\xybirational{\ar@{-->}}
\newcommand\xyquiverA{\ar@{->>}@/^/}

\renewcommand\Im{\operatorname{Im}}
\newcommand\Hom{\operatorname{Hom}}

\newcommand\RDerived{\mathbb{R}}

\newcommand\LDerived{\mathbb{L}}

\newcommand\Aut{\operatorname{Aut}}

\newcommand\End{\operatorname{End}}

\newcommand\pt{\operatorname{pt}}

\newcommand\Sym{\operatorname{Sym}}

\newcommand\D{\mathcal D}

\newcommand\mcE{\mathcal E}
\newcommand\mcF{\mathcal F}

\newcommand\mcT{\mathcal T}
\newcommand\mcW{\mathcal W}

\newcommand\mfX{\mathfrak X}

\newcommand{\cW}{\mathcal{W}}

\newcommand{\mfZ}{\mathcal{Z}}

\newcommand{\mfT}{\mathcal{T}}

\newcommand\quot{/\kern-.7ex/}
\newcommand\quotParam[1]{/\kern-.7ex/_{\kern-.4ex#1}}

\newcommand\beq[1]{\begin{equation}\label{#1}}
\newcommand\eeq{\end{equation}}
\newcommand\beqa{\begin{eqnarray*}}
\newcommand\eeqa{\end{eqnarray*}}

\opt{ams}{\theoremstyle{plain}}
\newtheorem{theorem}{Theorem}
\newtheorem{thm}[theorem]{Theorem}

\newtheorem{lemma}[theorem]{Lemma}

\newtheorem{cor}[theorem]{Corollary}

\newtheorem{prop}[theorem]{Proposition}

\opt{ams}{\theoremstyle{definition}}

\newtheorem{defn}[theorem]{Definition}

\opt{ams}{\theoremstyle{remark}}
\opt{ams}{\newtheorem*{acks}{Acknowledgements}}

\newtheorem{eg}[theorem]{Example}
\opt{ams}{\newtheorem*{notn}{Notation}}
\newtheorem{remark}[theorem]{Remark}

\numberwithin{theorem}{section}

\newtheorem{keythm}{Theorem}

\opt{ams}{}

\opt{lms}{ }

\newcommand{\emphasis}[1]{{\em #1}}

\opt{lms}{\setlength{\marginparwidth}{0.6in}}
\newcommand{\marginparstretch}{0.8}
\let\oldmarginpar\marginpar
\renewcommand\marginpar[1]{\-\oldmarginpar[\framebox{\setstretch{\marginparstretch}\begin{minipage}{\marginparwidth}{\raggedleft\footnotesize #1}\end{minipage}}]{\framebox{\setstretch{\marginparstretch}\begin{minipage}{\marginparwidth}{\raggedright\footnotesize #1}\end{minipage}}}}

\newcommand{\Gr}{\mathbbm{Gr}}
\newcommand{\quotStack}[2]{\left[ \frac{#1}{#2} \right]}
\newcommand{\spacedoplus}{\;\, \oplus \;\,}
\newcommand{\spacedtimes}{\;\, \times \;\,}
\newcommand{\row}[2]{\operatorname{row}_{#1}({#2})}
\newcommand{\col}[2]{\operatorname{col}_{#1}({#2})}

\newcommand{\PlainYoungDiag}[6]{

\def\xStart{#2} \def\yStart{#3}
\def\xSize{#4} \def\ySize{#5}

\draw (\xStart,\yStart) -- (\xStart+\xSize,\yStart);
\draw (\xStart,\yStart) -- (\xStart,\yStart+\ySize);
\draw[help lines] (\xStart,\yStart) grid (\xStart+\xSize,\yStart+\ySize);

\newcounter{#1row}
\foreach \rowLength in #6 {
  \ifnum\rowLength>0 {
    \foreach \col in {1,...,\rowLength}{
      \draw[fill=gray!50] (\xStart+\col-1,\yStart+\value{#1row}) rectangle (\xStart+\col,\yStart+\value{#1row}+1) ; } }
  \fi
  \stepcounter{#1row} }
}

\newcommand{\YoungDiag}[7]{

\def\xStart{#2} \def\yStart{#3}
\def\xSize{#4} \def\ySize{#5}
\draw   node at (\xStart+\xSize/2,\yStart-0.75) ()  {#7};
\PlainYoungDiag{#1}{#2}{#3}{#4}{#5}{#6}

}

\newcommand{\BibliographyLocation}{../../Bibliography}

\newcommand{\affiliationAddress}{School of Mathematics and Maxwell Institute of Mathematics, University of Edinburgh, Edinburgh, EH9 3JZ}
\newcommand{\affiliationCountry}{U.K.}
\newcommand{\affiliationEmail}{Will.Donovan@ed.ac.uk}

\begin{document}
\title[Grassmannian twists, derived equivalences and brane transport]{ Grassmannian twists, derived equivalences and brane transport }
\author{Will Donovan}

\opt{ams}{
\address{\affiliationAddress, \affiliationCountry}
\email{\affiliationEmail}

\thanks{\emph{Prepared for the proceedings of String-Math 2012, Bonn.} }
}

\opt{lms}{
\classno{Primary 14F05, 18E30; Secondary 14M15.}
\extraline{}
}

\begin{abstract}This note is based on a talk given at String-Math 2012 in Bonn, on a joint paper with Ed Segal \cite{donovan:2012wj}. We exhibit derived equivalences corresponding to certain Grassmannian flops. The construction of these equivalences is inspired by work of Herbst--Hori--Page on brane transport for gauged linear $\sigma$-models: in particular, we define `windows' corresponding to their grade restriction rules. We then show how composing our equivalences produces interesting autoequivalences, which we describe as twists and cotwists about certain spherical functors.

Our proofs use natural long exact sequences of bundles on Grassmannians known as twisted Lascoux complexes, or staircase complexes. We give a compact description of these. We also touch on some related developments, and work through some extended examples.
\end{abstract}

\maketitle

\section{Introduction}

It is conjectured \cite[Conjecture 5.1]{Kawamata:2002vq} that for flops between smooth projective complex varieties there exist corresponding \emphasis{derived equivalences}. In this note we discuss a particular set of examples: local models for \emphasis{Grassmannian flops}. These are obtained from geometric invariant theory (GIT) quotients of vector spaces by general linear group actions: the original variety corresponds to a certain stability condition, and the flopped variety to another.

The varieties occurring in our examples may be viewed as large-radius limits of particular gauged linear $\sigma$-models. Herbst, Hori and Page \cite{Herbst:2008wl} studied \emphasis{brane transport} between phases of such models, and produced a description in terms of grade-restriction rules. Inspired by this, we describe brane transport in our examples using \quotes{windows} corresponding to grade-restriction rules. More precisely, we construct equivalences between the bounded derived categories of coherent sheaves on the respective GIT quotients: see Theorem~\ref{keytheorem.equivalences}. By definition, these equivalences are \emphasis{functorial}, and so may be interpreted as describing transport of branes, and also of the strings between them. We then describe \emphasis{monodromy} for the brane transport, which turns out to be given by twists and cotwists about spherical functors: see Theorem~\ref{theorem.twist_cotwist}.

This note is intended to give a concise summary of the results in \cite{donovan:2012wj}, concentrating on examples, and referring to the original paper for proofs and technical details.

\subsection{Results}

We now describe our examples, and give precise statements of our findings. We fix throughout a dimension $d$. Then we have:

\begin{defn}Taking $V$ and $S^{(r)}$ to be complex vector spaces of dimension $d$ and $r$ respectively, we define a global quotient stack
$$\mfX^{(r)} := \quotStack{\Hom(S^{(r)},V)\spacedoplus \Hom(V,S^{(r)})}{\operatorname{GL}(S^{(r)})}.$$
(Here the $\operatorname{GL}(S^{(r)})$-action is the natural one, given by composition of linear maps.)
\end{defn}

We now consider two different GIT quotients, viewed as substacks of $\mfX^{(r)}$:

\begin{defn}
We define substacks $X_{\pm}^{(r)}$ of $\mfX^{(r)}$ by requiring that the first map be injective, or the second surjective, as shown below:
$$\begin{aligned} X_+^{(r)} & := \quotStack{\Hom_{\into} (S^{(r)},V) \spacedoplus \Hom(V,S^{(r)})}{\operatorname{GL}(S^{(r)}) } \\
X_-^{(r)} & := \quotStack{\Hom (S^{(r)},V)\spacedoplus \Hom_{\onto}(V,S^{(r)})}{ \operatorname{GL}(S^{(r)}) } \end{aligned} $$
\end{defn}

\begin{remark}\label{remark.geometry_of_quotients}The geometric interpretation of these spaces $X_{\pm}^{(r)}$ is as follows:
\begin{enumerate}
\item The GIT quotient $X^{(0)}_+$ is simply a point. 
\item The quotient $X^{(1)}_+$ is the total space of the bundle $l^{\oplus d}$ on $\PP V$, where we take $l$ to be tautological subspace bundle on $\PP V$ (which is often denoted by $\cO(-1)$).
\item In general, the quotient $X_+^{(r)}$ is the total space of the bundle $\Hom(V,S^{(r)})$ on the Grassmannian $\Gr(r,V)$, where we reuse the notation $S^{(r)}$ for the tautological subspace bundle on the Grassmannian. After choosing a basis for $V$, this is of course just the bundle $S^{(r) \oplus d}$. A dual construction gives $X_-^{(r)}$: see \cite[Section 3.1]{donovan:2012wj}. 
\end{enumerate}
\end{remark}

\begin{remark}The quotients $X^{(r)}_{\pm}$ should be thought of as corresponding to certain large-radius limits of a gauged linear $\sigma$-model determined by $\mfX^{(r)}$. See Section~\ref{section.physics} for some further remarks on this viewpoint.\end{remark}

We then have:

\begin{keythm}[ {\cite[Theorem 3.7]{donovan:2012wj}} ] \label{keytheorem.equivalences} For each $k \in \mathbb{Z}$, there exists an equivalence of bounded derived categories $$\psi_k : \D(X^{(r)}_+) \overset{\sim}{\To} \D(X^{(r)}_-).$$\end{keythm}

We consider now \emphasis{autoequivalences} produced by composing the equivalences $\psi_k$ and their inverses:

\begin{notn}We write $\omega_{k,l} := \psi_k^{-1}\psi^{\phantom{-1}}_l \!\!\!\! \in \Aut\left(\D(X^{(r)}_+)\right),$ and refer to these $\omega_{k,l}$ as {\em window-shift autoequivalences}.
\end{notn}

There is a well-developed theory associating twist autoequivalences of derived categories to \emphasis{spherical objects} \cite{Seidel:2000} or, more generally, to \emphasis{spherical functors} (for definitions, see Section~\ref{section.spherical}). These twists may be thought of as mirror to symplectic monodromies \cite{Seidel:2000,Thomas:2010tg}. The window-shift autoequivalences defined above can, at least heuristically, be seen as realising monodromy in the Stringy K\"ahler Moduli Space (SKMS) of our theory. Accordingly, the following result relates certain window-shift autoequivalences to twists of spherical functors. Other window-shifts are related to a dual notion of \emphasis{cotwists}:

\begin{keythm}[ {\cite[Theorems 3.12, 3.13]{donovan:2012wj}} ]\label{theorem.twist_cotwist} We have that \begin{align*}\omega_{0,+1} & = T_{F(r)}, \\
\omega_{0,-1}  & = C_{F(r+1)} [2(d-r)-1],\end{align*}
where the twist $T_{\bullet}$ and the cotwist $C_{\bullet}$ are defined in Section~\ref{section.spherical}, and each $F(r)$ is a certain spherical functor $$F(r): \D(X_+^{(r-1)}) \To \D(X_+^{(r)})$$ defined in Section~\ref{section.correspondences}.
\end{keythm}

\begin{remark}\phantom{a}
	
\begin{enumerate}
\item When $r=1$, the functor $F(1)$ has domain $\D(\pt)$ and we have $F(1)(-) = - \otimes \cO_{\PP V}.$ Here $\cO_{\PP V}$ is the skyscraper sheaf on the zero section of the bundle $X^{(1)}_+$, where we use the description of $X^{(1)}_+$ given in Remark~\ref{remark.geometry_of_quotients}. We then have that $T_{F(1)} = T_{\cO_{\PP V}}$, the spherical twist about the spherical object $\cO_{\PP V}$ \cite[Lemma 2.2]{Donovan:2011ua}.
\item For $r \geq 2$, the twist $T_{F(r)}$ has a more complicated geometric description. Generically, it acts as a family spherical twist \cite{Horja:2001}, but the behaviour on a certain closed locus is more elaborate. In the case $r=2$, this interesting locus is the zero section $\Gr(2,V)$ of $X_+^{(2)}$. See \cite{Donovan:2011vc} for discussion.
\end{enumerate}
\end{remark}

Finally note that we obtain a description of all window-shift equivalences $\omega_{k,l}$ in our examples as a corollary of Theorem~\ref{theorem.twist_cotwist}: see Section~\ref{section.window-shift_relations} for details.

\begin{figure}
\begin{tikzcd}[column sep=150pt, row sep=100pt, scale=3, every node/.style={scale=8}]
X_+^{(0)} \dar[swap]{F(1)} & \\
X_+^{(1)} \dar[swap]{F(2)}
\arrow[out=180-40, in=180-10, loop]{}[name=TF1,swap]{T_{F(1)}}
\arrow[out=40, in=10, loop]{}[name=w01-1]{\omega_{0,+1}}
\arrow[out=-40, in=-10, loop]{}[name=w0-1-1,swap]{\omega_{0,-1}}
\arrow[out=180+40, in=180+10, loop]{}[name=CF2]{C_{F(2)}}
\arrow[bend left=50]{r}[description]{\Phi_{+1}}
\arrow{r}[description]{\Phi_0}
\arrow[bend right=50]{r}[description]{\Phi_{-1}}
\arrow[-, dotted, to path={(TF1) ..controls +(0.5,0.1) and +(-0.5,0.15).. (w01-1)}]{}
\arrow[-, dotted, to path={(CF2) ..controls +(0.5,-0.1) and +(-0.5,-0.13).. (w0-1-1)}]{}
	& X_-^{(1)} \\ 
X_+^{(2)} \dar[swap]{F(3)}
\arrow[out=180-40, in=180-10, loop]{}[name=TF2,swap]{T_{F(2)}}
\arrow[out=40, in=10, loop]{}[name=w01-2]{\omega_{0,+1}}
\arrow[out=-40, in=-10, loop]{}[name=w0-1-2,swap]{\omega_{0,-1}}
\arrow[out=180+40, in=180+10, loop]{}[name=CF3]{C_{F(3)}}
\arrow[bend left=50]{r}[description]{\Phi_{+1}}
\arrow{r}[description]{\Phi_0}
\arrow[bend right=50]{r}[description]{\Phi_{-1}}
\arrow[-, dotted, to path={(TF2) ..controls +(0.5,0.1) and +(-0.5,0.15).. (w01-2)}]{}
\arrow[-, dotted, to path={(CF3) ..controls +(0.5,-0.1) and +(-0.5,-0.13).. (w0-1-2)}]{}
	& X_-^{(2)} \\
X_+^{(3)}
 & {\vdots}
\end{tikzcd}
\caption{Schematic of the functors involved in Theorem~\ref{theorem.twist_cotwist}, with arrows denoting functors between the derived categories of the respective spaces $X_{\pm}^{(r)}$. Autoequivalences of the $\D(X_+^{(r)})$ which coincide (up to a shift) are connected by dotted lines.}
\label{figure.schematic}
\end{figure}
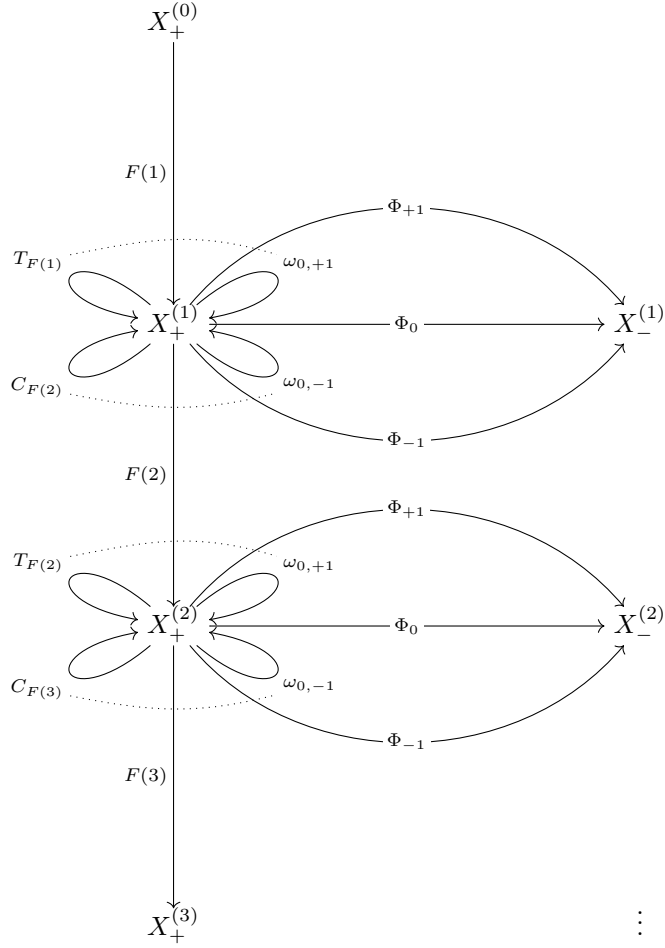

\newpage
\subsection{Outline}
\begin{itemize}
\item In Section~\ref{section.background}, we explain some background, and give a brief discussion of the physical interpretation of our results.
\item We construct our windows in Section~\ref{section.windows}, and indicate how they yield derived equivalences between GIT quotients (Theorem~\ref{keytheorem.equivalences}). 
\item In Section~\ref{section.twist_Lascoux}, we give a compact guide  to the \emphasis{twisted Lascoux resolutions} which are used in our proofs to produce long exact sequences of bundles on Grassmannians: these may also be of independent interest.
\item In Section~\ref{section.windows_and_twists}, we explain the links between window-shift autoequivalences and twists (Theorem~\ref{theorem.twist_cotwist}).
\item In Appendix~\ref{section.window_shift_examples}, we give explicit examples of the actions of window-shift autoequivalences.
\end{itemize}

\begin{acks} This note is based on a joint paper with Ed Segal \cite{donovan:2012wj}, which grew out of my thesis project suggested by Richard Thomas: their ideas and influence appear throughout, and this work would not have been possible without their generous support. I~also gratefully acknowledge the financial support of EPSRC during my doctoral work, and subsequently via grant EP/G007632 held in Edinburgh with Iain Gordon. Finally, I wish to thank the organisers of String-Math 2012 for an enjoyable and stimulating conference.\end{acks}

\begin{notn} We write:
\begin{itemize}
\item $\PP V$ for the projective space of lines in a vector space $V$, and $\PP^\vee V$ for the dual space of $1$-dimensional quotients;
\item $\Gr(r,V)$ for the Grassmannian of $r$-dimensional subspaces of $V$, and $\Gr(V,r)$ for the dual Grassmannian of $r$-dimensional quotients;
\item $\D(X)$ for the bounded derived category of coherent sheaves on a variety $X$;
\item $\delta$ for a Young diagram, or  the corresponding sequence of integers (Definition~\ref{defn.Young_diagrams});
\item $U^\delta$ for a Schur power of the vector space, or vector bundle, $U$ (Definition~\ref{defn.Schur_powers});
\item $\Hom_{\into}$ and $\Hom_{\onto}$ for injective and surjective maps respectively.
\end{itemize}
\end{notn}

\section{Background}
\label{section.background}
\subsection{Spherical functors}
\label{section.spherical}

Following \cite{Anno:2007wo, Anno:2010we} we have:

\begin{defn}\label{definition.spherical} Taking an integral functor $F : \D(Z) \to \D(X)$ with right adjoint $R$ we have: \begin{enumerate} \item a \textbf{twist} functor $T_F : \D(X) \to \D(X)$ defined so that 
$$ T_F (\mcE) := \cone{  FR(\mcE) \To \mcE  };$$
\item a \textbf{cotwist} functor $C_F: \D(Z) \to \D(Z)$ such that
$$ C_F (\mcE) := \cone{  \mcE \To RF(\mcE)  }.$$
\end{enumerate}
(The morphisms here are provided by the (co)unit of the adjunction.)
\end{defn}

In all our examples, $X$ and $Z$ are Calabi-Yau (see \cite[Section 3.2]{Donovan:2011ua}). Under this assumption, $T_F$ is an equivalence precisely when $C_F$ is an equivalence. In this latter circumstance, we refer to $F$ as \emphasis{spherical}. Note that there exists a more general definition which relaxes the Calabi-Yau assumptions: see \cite{Anno:2010we}.

\subsection{Related work concerning GIT}

Halpern-Leistner \cite{HalpernLeistner:2012uha} and Ballard, Favero, and Katzarkov \cite{Ballard:2012wia}, also employing ideas from \cite{Herbst:2008wl} and \cite{Segal:2009tua}, have developed a general theory of derived equivalences corresponding to certain variations of GIT: our equivalences $\psi_k$ fit in this framework. An upcoming paper \cite{HalpernLeistner:2013tp} relates window-shift autoequivalences for general variations of GIT involving a single Hesselink stratum to twists of spherical functors. Our results for $r=1$ follow from this general theory, so it is natural to ask whether the theory can be extended to cover the cases $r>1$. We hope that this will be the subject of future work.

\subsection{Links with gauged linear $\sigma$-models}
\label{section.physics}

As has been noted, the quotients $X^{(r)}_{\pm}$ should be thought of as corresponding to certain large-radius limits of a gauged linear $\sigma$-model determined by $\mfX^{(r)}$, defined from the data of the underlying vector space of $\mfX^{(r)}$ with its $\operatorname{GL}(r)$-action: see \cite{donovan:2012wj} for a more detailed account.

In physical language, our model has a $\operatorname{U}(r)$ gauge group coupled to $d$ chiral multiplets $\Phi_i$ transforming in the fundamental representation $\mathbf{r}$, with another $d$ fields transforming in the anti-fundamental representation $\bar{\mathbf{r}}$ (cf. \cite[Section 2.1]{Hori:2006hv}): these latter fields ensure that the model is Calabi-Yau. Setting $d=2$ and $r=1$, we recover the brane transport analysis of \cite[Section 8.4.1]{Herbst:2008wl}. Note that each of our spherical functors $F(r)$ relates the categories of B-branes associated to a pair of GLSMs, with gauge groups $\operatorname{U}(r)$ and $\operatorname{U}(r+1)$ respectively.

There has been some discussion \cite{Aspinwall:2002tu,Herbst:2009wi,HalpernLeistner:2013tp} of how certain SKMS monodromies give rise to family spherical twists, or `$EZ$ twists' \cite{Horja:2001}. These twists may be described in the framework of Section~\ref{section.spherical}: the associated cotwist is given by a composition of a shift, and tensoring by a line bundle. In our examples, by contrast, the cotwist is no longer of this form: for $r=2$, it is itself a twist about a spherical object, whereas for $r>2$ it is even more complicated \cite{donovan:2012wj}. It would be interesting to investigate how these exotic twists arise from monodromy. Furthermore, Theorem~\ref{theorem.twist_cotwist} suggests that, in our examples, SKMS monodromies yield both a twist and a \emphasis{cotwist}: we do not know how this should be interpreted physically.

\section{Windows}
\label{section.windows}
\subsection{Schur powers}
\begin{notn}We write $S^{(r)}$ for the tautological rank $r$ bundle on the stack $\mfX^{(r)}$, or simply $S$ when the $r$ is clear from context. 
\end{notn}

We describe how to obtain various other natural bundles on $\mfX^{(r)}$ by the \emphasis{Schur power} construction. These bundles will be used as generators for our windows. They are indexed by arrangements of boxes known as \emphasis{Young diagrams}:

\begin{defn}[Young diagrams]\label{defn.Young_diagrams} Given a finite non-decreasing sequence $\delta = ( \delta_1, \delta_2, \ldots, \delta_h) $ of non-negative integers $\delta_i$ there is a corresponding Young diagram. This is given by a stacked arrangement of boxes, with the sequence $\delta$ indicating the number of boxes contained in each successive layer (see Example~\ref{eg.young_diag} below).
\end{defn}

\begin{notn}We write $\row{i}{\delta}$ and $\col{i}{\delta}$ for the length of the $i^{\text{th}}$ row and $i^{\text{th}}$ column of the diagram $\delta$, respectively.\end{notn}

\begin{figure}[H]
\begin{center}\begin{tikzpicture}[auto,>=latex',scale=0.55]
\PlainYoungDiag{delta}{0}{0}{3.5}{2.5}{{3,1}}
\end{tikzpicture}\end{center}
\caption{Young diagram $\delta=(3,1)$.}
\label{figure.young_diagram_example}
\end{figure}
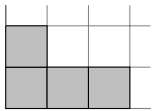

\begin{eg}\label{eg.young_diag} The sequence $\delta=(3,1)$ corresponds to the Young diagram in Figure~\ref{figure.young_diagram_example} above. By definition, we have $\row{i}{\delta}=\delta_i$. Observe also that $$\col{i}{\delta}=\begin{cases} 2 & i=1, \\ 1 & i=2,3, \\ 0 & i \geq 4. \end{cases}$$
\end{eg}

\begin{defn}[Schur powers]\label{defn.Schur_powers} Given a vector space $U$, we write $ U^\delta$ for the irreducible $\operatorname{GL}(U)$-representation with highest weight $\delta$ \cite{Fulton:1996tk}.
\end{defn}

\begin{eg}We have that \begin{equation*}\begin{gathered} U^{(1,\ldots,1)} \iso \wedge^h U, \\ U^{(w)} \iso \Sym^w U. \end{gathered}\end{equation*} More general Schur powers have a more complicated description.
\end{eg}

\begin{remark} Applying the same Schur power construction relative to a base, we may take Schur powers of \emphasis{vector bundles} also. In particular, we obtain natural bundles $S^{\vee \delta}$ on the stacks $\mfX^{(r)}$, which we use in the following Section \ref{section.window_def}.\end{remark}

\subsection{Equivalences}
\label{section.window_def}
We may now construct our windows:
\begin{defn}[windows]We define the window $\cW_k$ for $k \in \Z$ to be the full subcategory of $\D(\mfX^{(r)})$ split-generated by the following set of vector bundles:

$$W_k= \setconds{ S^{\vee \delta} \otimes \det(S^\vee)^{\otimes k} }{ \row{\bullet}{\delta} \leq d-r, \: \col{\bullet}{\delta} \leq r }.$$

\end{defn}

\begin{remark}We choose the same representations  as in Kapranov's exceptional collection for the Grassmannian \cite{Kapranov:2009wf}. As a consequence we find \cite[Appendix C]{Donovan:2011ua} that the restrictions of the bundles $W_k$, to either $X_+$ or $X_-$, sum to give a \emphasis{tilting bundle}. This is the crucial point in the proof of the following lemma, which we omit:
\end{remark}

\begin{lemma}[ {\cite[Proposition 3.6]{donovan:2012wj}} ]\label{lemma.splitting} Writing $i_{\pm}$ for the inclusions $X_\pm \into \mfX$, the (derived) restriction functors $\LDerived i^*_{\pm }$ give equivalences $$\LDerived i^*_{\pm } : \cW_k \overset{\sim}{\To} \D(X_\pm).$$ \end{lemma}

Using this lemma we have a commutative diagram as follows, which defines the required equivalences $\psi_k$:
$$\begin{tikzcd}[column sep=30pt, row sep=0pt]
& \D(\mfX) \arrow{ddl}[swap]{\LDerived i^*_+} \arrow{ddr}{\LDerived i^*_-}& \\
& \cup & \\
  \D(X_+) \arrow[bend right=30]{rr}{\sim}[swap]{\psi_k} & \mcW_k \lar[swap]{\sim} \rar{\sim} & \D(X_-)
\end{tikzcd}$$

\begin{remark}The equivalences $\psi_k$ are obtained in \cite[Section 5]{Buchweitz:2011ug} using a different method, which is characteristic-free.\end{remark}

\subsection{Examples}

It is an exercise \cite{Fulton:1996tk} that if we take $\delta$ of length $r$ then in fact $$S^{\vee \delta} \otimes \det(S^\vee) \iso  S^{\vee \delta'}$$ where $\delta'$ is obtained by incrementing each element of the integer sequence $\delta$. We deduce that windows $\cW_k$ with consecutive $k$ have a substantial overlap. Examples are as follows:

\newcommand\lineNodes[7]{
\def\factor{#1}
\def\xshift{#2}
\def\padding{1/2 * \factor}
\node ({#3}Schur3) at (-6 * \factor + \xshift, 0) {#4};
\node ({#3}Schur2) at (-4 * \factor + \xshift, 0) {#5};
\node ({#3}Schur1) at (-2 * \factor + \xshift,0) {#6};
\node ({#3}Schur0) at (0 + \xshift,0) {#7};
\draw[gray] (0 + \xshift + 3/2*\padding ,\padding) -- (-6 * \factor + \xshift - 3/2*\padding , \padding) -- (-6 * \factor + \xshift - 3/2*\padding , -\padding) -- (0 + \xshift + 3/2*\padding ,-\padding) -- cycle;
}

\newcommand\twoLineNodes[4]{
\def\factor{#1}
\def\xshift{#2}
\node ({#3}Schur4) at (-8 * \factor + \xshift, 0) {#4};
\draw[gray] (-2 + \xshift + 3/2*\padding ,\padding) -- (-8 * \factor + \xshift - 3/2*\padding , \padding) -- (-8 * \factor + \xshift - 3/2*\padding , -\padding) -- (-2 + \xshift + 3/2*\padding ,-\padding) -- cycle;}

\newcommand\twoLineWindowNodes[1]{
\lineNodes{#1}{0}{window}{$S^{\vee (d-1)}$}{$\ldots$}{$S^{\vee (1)}$}{$S^{\vee (0)}$}
\twoLineNodes{#1}{0}{window}{$S^{\vee (d)}$}
}

\newcommand\triangleNodes[9]{
\def\factor{#1}
\def\xshift{#2}
\def\padding{2/3 * \factor}
\node ({#3}Schur22) at (-4 * \factor + \xshift, 0) {#4};
\node (above{#3}Schur22) at (-4 * \factor + \xshift, 1) {};
\node ({#3}Schur11) at (-2 * \factor + \xshift, 0) {#5};
\node ({#3}Schur00) at (0 + \xshift,0) {#6};
\node (above{#3}Schur00) at (0 + \xshift, 1) {};
\node ({#3}Schur21) at (-3 * \factor + \xshift, 1 * \factor) {#7};
\node ({#3}Schur10) at (-1 * \factor + \xshift, 1 * \factor) {#8};
\node ({#3}Schur20) at (-2 * \factor + \xshift, 2 * \factor) {#9};
\draw[gray] (0 + \xshift +2*\padding,-2/3*\padding) -- (-4 * \factor + \xshift -2*\padding, -2/3*\padding) -- (-2 * \factor + \xshift, 2 * \factor + 4/3*\padding) -- cycle;
}

\newcommand\twoTriangleNodes[6]{
\def\factor{#1}
\def\xshift{#2}
\node ({#3}Schur33) at (-6 * \factor + \xshift, 0) {#4};
\node ({#3}Schur32) at (-5 * \factor + \xshift, 1 * \factor) {#5};
\node ({#3}Schur31) at (-4 * \factor + \xshift, 2 * \factor) {#6};
\draw[gray] (-2 * \factor + \xshift+2*\padding,-2/3*\padding) -- (-6 * \factor + \xshift-2*\padding,-2/3*\padding) -- (-4 * \factor + \xshift, 2 * \factor + 4/3*\padding) -- cycle;
}

\newcommand\twoWindowNodes[1]{
\triangleNodes{#1}{0}{window}{$S^{\vee (2,2)}$}{$S^{\vee (1,1)}$}{$S^{\vee (0,0)}$}{$S^{\vee (2,1)}$}{$S^{\vee (1,0)}$}{$S^{\vee (2,0)}$}
\twoTriangleNodes{#1}{0}{window}{$S^{\vee (3,3)}$}{$S^{\vee (3,2)}$}{$S^{\vee (3,1)}$}
}

\begin{eg}
\label{eg.P_windows}
We illustrate the windows $\cW_{+1}$ and $\cW_0$ for $r=1$:

\begin{figure}[H]
\begin{center}
\begin{tikzpicture}
\twoLineWindowNodes{1}

\node (W1) at (-10,1) {$\cW_{+1}$};
\node (W0) at (2,1) {$\cW_0$};
\draw [thick,->] (W1) -- ({window}Schur4);
\draw [thick,->] (W0) -- ({window}Schur0);

\end{tikzpicture}
\end{center}
\caption{Windows for $r=1$.}
\end{figure}
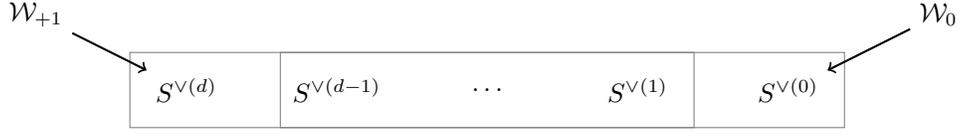

Note that in this case $S=l$, say, a line bundle. We hence have simply that $S^{\vee (d)} = l^{\vee d}$, and so we recover the Beilinson tilting bundles on $X^{(1)}_{\pm}$.
\end{eg}

\begin{eg}We give windows $\cW_{+1}$ and $\cW_0$ for the Grassmannian example $d=4$, $r=2$:

\label{eg.Gr_windows}
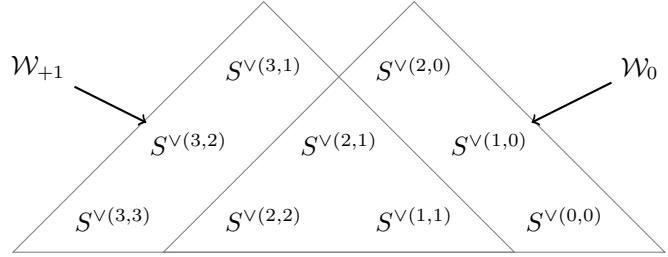
\begin{figure}[H]
\begin{center}
\begin{tikzpicture}
\twoWindowNodes{1}

\node (W1) at (-7,2) {$\cW_{+1}$};
\node (W0) at (1,2) {$\cW_0$};
\draw [thick,->] (W1) -- ({window}Schur32);
\draw [thick,->] (W0) -- ({window}Schur10);

\end{tikzpicture}
\end{center}
\caption{Windows for $d=4$, $r=2$.}
\end{figure}
\end{eg}

\subsection{Relations between window shifts}
\label{section.window-shift_relations}

The following lemma, which follows immediately from the definitions, may be used to express a general window-shift autoequivalence $\omega_{k,l}$ in terms of those described in Theorem~\ref{theorem.twist_cotwist}:

\begin{lemma}
\begin{equation*}\begin{gathered}\omega_{m,k} \circ \omega_{k,l} = \omega_{m,l} \\ \omega_{k+m,l+m} = (- \otimes \det(S^\vee)^{\otimes m})\circ \omega_{k,l} \circ (-\otimes \det(S)^{\otimes m}) \end{gathered}\end{equation*}
\end{lemma}

\section{Twisted Lascoux complexes}
\label{section.twist_Lascoux}
We explain a case of a certain generalised Koszul resolution known as the \emphasis{twisted Lascoux resolution} \cite[Section 6.1]{Weyman:2007wu}. Weyman gives this resolution implicitly in {\it loc. cit.}: we perform a Borel--Weil--Bott calculation which makes it explicit in \cite[Theorem A.7]{donovan:2012wj}. We summarise the results here, and give examples.

\begin{remark} A very closely-related construction appears in \cite{Fonarev:2011tq} under the name of \emphasis{staircase complexes}.\end{remark}

We begin by defining a stack $$\mfT := \quotStack{\Hom(V, S^{(r)})}{\operatorname{GL}(S^{(r)})}.$$ 
Note that the dual Grassmannian $\Gr(V,r)$ arises as a substack of $\mfT$ by restricting to the locus of surjective $\Hom$s.  The following Theorem~\ref{theorem.freeresolutions} resolves certain natural torsion sheaves \eqref{equation.resolved_sheaf}, supported on the complement of this locus. Corollary~\ref{cor.sequences} then produces associated exact sequences on the Grassmannian.

We consider then the composition morphism $j$ given as follows
$$ j: \quotStack{\Hom(V, S^{(r-1)}) \spacedoplus \Hom_{\into} (S^{(r-1)}, S^{(r)}) }{\operatorname{GL}(S^{(r-1)}) \spacedtimes \operatorname{GL}(S^{(r)})} \To \mfT, $$
and the direct images of certain Schur powers $(S^{(r-1) \vee})^{\delta}$ under $j$. We have:
\begin{thm}[ { \cite[Theorem A.7]{donovan:2012wj} } ]\label{theorem.freeresolutions}
Let $\delta$ be a Young diagram with \begin{equation}\label{equation.lascoux_shape_restriction} \row{\bullet}{\delta} \leq d-r+1, \quad\col{\bullet}{\delta} < r.\end{equation} Then \begin{equation}\label{equation.resolved_sheaf} j_* \left((S^{(r-1) \vee})^{\delta}\right)\end{equation} may be resolved by a complex $\mcE_\bullet$ of bundles where \begin{align*} \mcE_0 & := (S^{(r) \vee})^{ \delta}, \\ \mcE_k & := (S^{(r) \vee})^{ \delta_k} \otimes \wedge^{s_k} V. \end{align*}
The Young diagrams $\delta_k$ are defined, for $1 \leq k \leq d-r+1$, such that
\begin{equation}\label{equation.diagram_algorithm} \row{i}{\delta_k} = \begin{cases} \row{i}{\delta} & 1 \leq i < \col{k}{\delta}+1 \\k & i = \col{k}{\delta} + 1 \\ \row{i-1}{\delta}+1 & \col{k}{\delta} +1 < i \leq r \end{cases}\end{equation}
and $s_k:=r+k-(\col{k}{\delta}+1)$.
\end{thm}

The following corollary then yields exact sequences on $\Gr(V,r)$:

\begin{cor}\label{cor.sequences}
Take $\delta$ a Young diagram with shape as above in \eqref{equation.lascoux_shape_restriction}. Then we have an exact sequence on $\Gr(V,r)$, as follows:
\begin{equation*} \begin{tikzcd}[column sep=14pt] 0 \rar & S^{\vee \delta_{K}} \otimes \wedge^{s_{K}} V  \rar &\;   \ldots\;  \rar & S^{\vee \delta_1} \otimes \wedge^{s_1} V  \rar & S^{\vee \delta} \rar &  0, \end{tikzcd} \end{equation*}
where
\begin{itemize}
\item the $\delta_k$ are defined as previously by \eqref{equation.diagram_algorithm},
\item $s_k := r+k-(\col{k}{\delta}+1)$ as above,
\item $K := d-r+1$, and
\item $S$ denotes the tautological quotient bundle on $\Gr(V,r)$.
\end{itemize}
\end{cor}
\begin{proof}This follows immediately, as the support of any object in $\Im(j_*)$, and in particular the sheaf \eqref{equation.resolved_sheaf}, lies entirely in the locus which is removed from $\mcT$ to give the Grassmannian $\Gr(V,r)$.
\end{proof}

\begin{eg}We illustrate in Figure~\ref{figure.algorithm} below the Young diagrams $\delta_k$ appearing in Theorem~\ref{theorem.freeresolutions} and Corollary~\ref{cor.sequences} above, as defined by \eqref{equation.diagram_algorithm}, taking $\delta=(3,1)$, and $d=7$, $r=3$.
\end{eg}

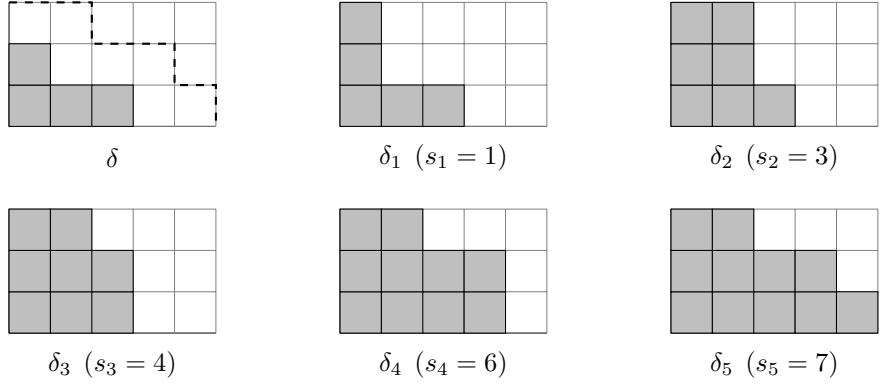
\begin{figure}[H]
\begin{tikzpicture}[auto,>=latex',scale=0.55]
\YoungDiag{delta_0}{0}{0}{5}{3}{{3,1}}{$\delta$}
\YoungDiag{delta_1}{8}{0}{5}{3}{{3,1,1}}{$\delta_1\:\:(s_1=1)$}
\YoungDiag{delta_2}{16}{0}{5}{3}{{3,2,2}}{$\delta_2\:\:(s_2=3)$}
\YoungDiag{delta_3}{0}{-5}{5}{3}{{3,3,2}}{$\delta_3\:\:(s_3=4)$}
\YoungDiag{delta_4}{8}{-5}{5}{3}{{4,4,2}}{$\delta_4\:\:(s_4=6)$}
\YoungDiag{delta_5}{16}{-5}{5}{3}{{5,4,2}}{$\delta_5\:\:(s_5=7)$}

\draw[dashed, line width=0.8pt] (0,3) -- (2,3) -- (2,2) -- (4,2) -- (4,1) -- (5,1) -- (5,0);
\end{tikzpicture}
\caption{Young diagrams $\delta_k$ for $\delta=(3,1)$, with $d=7$, $r=3$. The dashed line, in the diagram for $\delta$, indicates where boxes are added to produce each $\delta_k$.}
\label{figure.algorithm}
\end{figure}

\begin{eg}\label{eg.long_euler_sequence}Taking $\delta=(0)$ with $r=1$ we immediately see that $\delta_k=(k)$. Corollary~\ref{cor.sequences} then recovers the following well-known exact sequences on the projective space of quotients $\PP^\vee V$, where $l$ denotes the tautological quotient bundle:
\begin{equation*} \begin{tikzcd}[column sep=14pt] 0 \rar & l^{\vee d} \otimes \wedge^d V  \rar &\;   \ldots\;  \rar & l^{\vee} \otimes V  \rar &   \cO \rar &  0 \end{tikzcd} \end{equation*}
\end{eg}

\begin{eg}\label{eg.grassmannian_sequences} (cf. \cite[Examples A.8, A.9]{donovan:2012wj}) Taking $d=4$ and $r=2$, Corollary~\ref{cor.sequences} gives exact sequences
$$ \begin{gathered} 
\begin{tikzcd}[column sep=14pt] 0 \rar & S^{\vee (3,1)} \otimes \wedge^4 V \rar & S^{\vee (2,1)} \otimes \wedge^3 V \rar & S^{\vee (1,1)} \otimes \wedge^2 V \rar & S^{\vee (0,0)} \rar & 0 \end{tikzcd} \\
\begin{tikzcd}[column sep=14pt]0 \rar &  S^{\vee (3,2)} \otimes \wedge^4 V \rar & S^{\vee (2,2)} \otimes \wedge^3 V \rar & S^{\vee (1,1)} \otimes V \rar & S^{\vee (1,0)} \rar & 0 \end{tikzcd} \\
\begin{tikzcd}[column sep=14pt]0 \rar &  S^{\vee (3,3)} \otimes \wedge^4 V \rar & S^{\vee (2,2)} \otimes \wedge^2 V \rar & S^{\vee (2,1)}  \otimes V \rar & S^{\vee (2,0)} \rar & 0 \end{tikzcd} \end{gathered} $$
with $\delta=(0,0)$, $(1,0)$ and $(2,0)$ respectively.
\end{eg}

\begin{remark}Notice that each of these sequences lies in two consecutive windows: all but the last term lie in $\cW_{+1}$, and all but the first lie in $\cW_0$. This observation will be key in what follows.
\end{remark}

\section{Windows and twists}
\label{section.windows_and_twists}
\subsection{Transfer functors}

\begin{defn}[transfer functor]\label{defn.transfer_functor} A functor $\Phi \in \End(\D (\mathcal{X}))$ which
\begin{enumerate}
\item restricts to a functor $\cW_k \to \cW_l$, and
\item such that the following diagram commutes
\begin{equation}\label{equation.transfer_functor_diagram} \begin{tikzcd}
\D(X^+) \dar[swap]{\phi} & \cW_k \lar{\sim}[swap]{\LDerived i^*_+} \rar{\LDerived i^*_-}[swap]{\sim} \dar{\Phi} & \D(X^-) \dar[equals] \\
\D(X^+) & \D (\mathcal{X}) \lar[swap]{\LDerived i^*_+} \rar{\LDerived i^*_-}  & \D(X^-)
\end{tikzcd} \end{equation}
for some endofunctor $\phi \in \End(\D (X^+))$,
\end{enumerate}
is referred to as a \emphasis{transfer functor} for $\phi$.
\end{defn}

We think of $\Phi$ as transferring from window $\cW_k$ to $\cW_l$ in a manner compatible with the endomorphism $\phi$ on $\D(X^+)$. The following proposition then allows us to prove Theorem~\ref{theorem.twist_cotwist}, by constructing suitable transfer functors in our examples:

\begin{prop}Given a transfer functor $\Phi$ for $\phi$ as in Definition~\ref{defn.transfer_functor} above, we have $$\omega_{l,k} = \phi.$$
\end{prop}
\begin{proof}This follows formally. See \cite[Proposition 2.2]{donovan:2012wj}.
\end{proof}

\subsection{Hecke correspondences}
\label{section.correspondences}
We now introduce natural correspondences which may be used to construct transfer functors in our examples. We put:
$$\mfZ := \quotStack{\Hom(S^{(r)}, V) \spacedoplus \Hom(V, S^{(r-1)}) \spacedoplus \Hom_{\into}(S^{(r-1)}, S^{(r)})}{\operatorname{GL}(S^{(r-1)})\spacedtimes \operatorname{GL}(S^{(r)}) }.$$
Observe that $\mfZ$ is equipped with natural morphisms to $\mfX^{(r-1)}$ and $\mfX^{(r)}$, given by composition of linear maps, which we denote as follows:
$$\mfX^{(r-1)} \overset{\pi}{\longleftarrow} \mfZ \overset{j}{\longrightarrow} \mfX^{(r)}.$$
By passing through the correspondence $\mfZ$, we then obtain a functor $$\mathcal{F}(r) : \D(\mfX^{(r-1)}) \To \D(\mfX^{(r)}),$$ defined by $\mathcal{F}(r) := \RDerived j_* \circ \LDerived \pi^*$.

If we replace $\Hom(S^{(r)}, V)$ in the definition of $\mfZ$ with $\Hom_{\into}\left(S^{(r)},V\right)$, then we obtain a space, say $Z_+$, now with natural morphisms to $X^{(r-1)}_+$ and $X^{(r)}_+$. By passing through the correspondence $Z_+$, we obtain a functor $$F(r): \D(X_+^{(r-1)}) \To \D(X_+^{(r)}).$$

Fixing now a particular $r$, and dropping this from the notation for simplicity, we have:

\begin{prop}[ { \cite[Section 3.2.2]{donovan:2012wj} } ]
\label{prop.twist_transfer_functor}
$T_{\mathcal{F}}$ is a transfer functor for $T_F$ mapping window $\cW_{+1}$ to $\cW_0$.
\end{prop}

We say a few words about the proof of Proposition \ref{prop.twist_transfer_functor} here:

\begin{enumerate}
\item The proof that $T_{\mathcal{F}}$ maps window $\cW_{+1}$ to $\cW_0$
is given in \cite[Lemma 3.16]{donovan:2012wj}. The core of the proof is that for a generator $\mcE \in W_{+1}$ (as in Example~\ref{eg.Gr_windows}) we have
$$\mathcal{F}\mathcal{R} (\mcE) = \begin{cases} 0 & \mcE \in W_0, \\ \{\mcE \To \ldots \} & \mcE \not\in W_0, \end{cases}$$
where the omitted part of the complex is given by elements of $W_0$ (tensored by certain fixed vector spaces), and the generator $\mcE$ appears in homological degree $0$. This all follows using the Lascoux resolutions of Section~\ref{section.twist_Lascoux}. We hence deduce that $T_{\mathcal{F}} (\mcE) \in \cW_0$ as required. 
\item The commutativity of the left-hand square in diagram \eqref{equation.transfer_functor_diagram} is \cite[Lemma 3.17]{donovan:2012wj}: it follows from the formal similarity between $T_{\mcF}$ and $T_F$, and certain vanishing results. The commutativity of the right-hand square follows immediately from considering the support of objects in $\Im(\mathcal{F})$: by construction, they lie entirely in the locus which we remove to yield $X_-$. \end{enumerate}

\begin{remark} The cotwist proof is similar, though somewhat more delicate: see \cite[Section 3.2.3]{donovan:2012wj}.\end{remark}

\appendix

\section{Window shift examples}
\label{section.window_shift_examples}

Finally, we give some examples of the action of window-shift autoequivalences, consistent with the twist action described in the previous Section~\ref{section.windows_and_twists}. Complexes surrounded by braces denote objects of the derived category, with the left-most term in homological degree $0$.

\begin{eg}[$r=1$]Example~\ref{eg.P_windows} gives that the window $\cW_{+1}$ has generators $\mathcal{E} = l^{\vee k}$ for $1 \leq k \leq d$. From the definitions we have that $\psi_{+1} \left( l^{\vee k} \right) \iso l^{\vee k}$ and thence we obtain:

$$\omega_{0,+1} \left( l^{\vee k} \right) = \begin{cases} l^{\vee k} & k = 1, \ldots, d-1, \\ \{ \begin{tikzcd}[column sep=14pt] l^ {\vee d-1} \otimes \wedge^{d-1} V \rar &\;   \ldots\;  \rar & l^{\vee} \otimes V  \rar &  \cO \end{tikzcd} \} & k=d. \end{cases}$$

The only non-trivial part of this calculation is to use the exact sequence of Example~\ref{eg.long_euler_sequence} on $X^{-}$ (we have to pull it up from $\PP^\vee V$, and choose an isomorphism $\det V \iso \C$) to express $l^{\vee d}$ as a complex of vector bundles in the window $\mcW_0$. See \cite[{Section 2.1.1}]{donovan:2012wj} for the case $d=2$.
\end{eg}

\begin{eg}[$d=4$, $r=2$]Example~\ref{eg.Gr_windows} gives generators $\mcE = S^{\vee \delta}$ for the window $\cW_{+1}$ with certain $\delta$. For these $\delta$ we have:
$$\omega_{0,+1} \left( S^{\vee \delta} \right) = \begin{cases}  S^{\vee \delta} & \delta = (1,1), \: (2,2), \: (2,1), \\
\{ \begin{tikzcd}[column sep=14pt] S^{\vee (2,2)} \otimes \wedge^2 V \rar & S^{\vee (2,1)}  \otimes V \rar & S^{\vee (2,0)} \end{tikzcd} \} & \delta=(3,3), \\
\{ \begin{tikzcd}[column sep=14pt] S^{\vee (2,2)} \otimes \wedge^3 V \rar & S^{\vee (1,1)} \otimes V \rar & S^{\vee (1,0)} \end{tikzcd} \} & \delta=(3,2), \\
\{ \begin{tikzcd}[column sep=14pt] S^{\vee (2,1)} \otimes \wedge^3 V \rar & S^{\vee (1,1)} \otimes \wedge^2 V \rar & S^{\vee (0,0)} \end{tikzcd} \} & \delta=(3,1). \end{cases}$$

The calculation proceeds as above, this time using the exact sequences in Example~\ref{eg.grassmannian_sequences}. Similar calculations (with a twist) may be found in \cite[{Section 2.1.2}]{donovan:2012wj}.

\end{eg}


\bibliographystyle{\BibliographyLocation/Styles/halpha}

\newpage
\begin{thebibliography}{BLVdB11}


\bibitem[AL10]{Anno:2010we}
Rina Anno and Timothy Logvinenko.
\newblock {Orthogonally spherical objects and spherical fibrations}.
\newblock \href{http://arxiv.org/abs/1011.0707}{arXiv:1011.0707}.

\bibitem[Ann07]{Anno:2007wo}
Rina Anno.
\newblock {Spherical functors}.
\newblock \href{http://arxiv.org/abs/0711.4409}{arXiv:0711.4409}.

\bibitem[AHK05]{Aspinwall:2002tu}
Paul~S Aspinwall, Richard~Paul Horja, and Robert~L Karp.
\newblock {Massless D-Branes on Calabi-Yau Threefolds and Monodromy}.
\newblock {\em Communications in Mathematical Physics}, 259(1):45--69, October
  2005, \href{http://arxiv.org/abs/hep-th/0209161}{arXiv:hep-th/0209161}.


\bibitem[BFK12]{Ballard:2012wia}
Matthew Ballard, David Favero, and Ludmil Katzarkov.
\newblock {Variation of geometric invariant theory quotients and derived
  categories}.
\newblock \href{http://arxiv.org/abs/1203.6643}{arXiv:1203.6643}.

\bibitem[BLVdB11]{Buchweitz:2011ug}
Ragnar-Olaf Buchweitz, Graham~J Leuschke, and Michel Van~den Bergh.
\newblock {Non-commutative desingularization of determinantal varieties, II}.
\newblock \href{http://arxiv.org/abs/1106.1833}{arXiv:1106.1833}.


%
%
%

\bibitem[Don11a]{Donovan:2011vc}
Will Donovan.
\newblock {\em {Grassmannian twists on derived categories of coherent
  sheaves}}.
\newblock PhD thesis, Imperial College London, July 2011,
  \url{http://www.maths.ed.ac.uk/willdonovan/documents/thesis/twist.pdf}

\bibitem[Don11b]{Donovan:2011ua}
Will Donovan.
\newblock {Grassmannian twists on the derived category via spherical functors}.
\newblock {\em Proceedings of the London Mathematical Society} (to appear),
  \href{http://arxiv.org/abs/1111.3774}{arXiv:1111.3774}.
 
\bibitem[DS12]{donovan:2012wj}
Will Donovan and Ed Segal, 
\newblock {Window shifts, flop equivalences and
  Grassmannian twists}.
\newblock \href{http://arxiv.org/abs/1206.0219}{arXiv:1206.0219}.

%

\bibitem[FH96]{Fulton:1996tk}
William Fulton and Joe Harris.
\newblock {\em {Representation theory: a first course‎}}, volume 129 of {\em
  Graduate Texts in Mathematics}.
\newblock Springer, 1996.

\bibitem[Fon11]{Fonarev:2011tq}
Anton Fonarev.
\newblock {On minimal Lefschetz decompositions for Grassmannians}.
\newblock \href{http://arxiv.org/abs/1108.2292}{arXiv:1108.2292}.

\bibitem[HHP08]{Herbst:2008wl}
Manfred Herbst, Kentaro Hori, and David Page.
\newblock {Phases Of $\mathcal{N}=2$ Theories In $1+1$ Dimensions With
  Boundary}.
\newblock \href{http://arxiv.org/abs/0803.2045}{arXiv:0803.2045}.

\bibitem[HL12]{HalpernLeistner:2012uha}
Daniel Halpern-Leistner.
\newblock {The derived category of a GIT quotient}.
\newblock \href{http://arxiv.org/abs/1203.0276}{arXiv:1203.0276}.

\bibitem[HLS13]{HalpernLeistner:2013tp}
Daniel Halpern-Leistner and Ian Shipman.
\newblock {Autoequivalences of derived categories via geometric invariant
  theory}.
\newblock \href{http://arxiv.org/abs/1303.5531}{arXiv:1303.5531}.


\bibitem[Hor05]{Horja:2001}
Richard~Paul Horja.
\newblock {Derived Category Automorphisms from Mirror Symmetry}.
\newblock {\em Duke Mathematics Journal}, 127(1):1--34, 2005,
  \href{http://arxiv.org/abs/math/0103231}{arXiv:math/0103231}.

\bibitem[HT06]{Hori:2006hv}
Kentaro Hori and David Tong.
\newblock {Aspects of Non-Abelian Gauge Dynamics in Two-Dimensional
  $\mathcal{N}=(2,2)$ Theories}.
\newblock \href{http://arxiv.org/abs/hep-th/0609032}{arXiv:hep-th/0609032}.

\bibitem[HW12]{Herbst:2009wi}
Manfred Herbst and Johannes Walcher.
\newblock {On the unipotence of autoequivalences of toric complete intersection Calabi-Yau categories}.
\newblock {\em Mathematische Annalen}, 353(3):783--802, July 2012,
  \href{http://arxiv.org/abs/0911.4595}{arXiv:0911.4595}.

\bibitem[Kap88]{Kapranov:2009wf}
Mikhail Kapranov.
\newblock {On the derived categories of coherent sheaves on some homogeneous
  spaces}.
\newblock {\em Inventiones Mathematicae}, 92:479--508, 1988.

\bibitem[Kaw02]{Kawamata:2002vq}
Yujiro Kawamata.
\newblock {$D$-equivalence and $K$-equivalence}.
\newblock {\em Journal of Differential Geometry}, 61(1):147--171, 2002.

\bibitem[Seg11]{Segal:2009tua}
Ed~Segal.
\newblock {Equivalences between GIT quotients of Landau-Ginzburg B-models}.
\newblock {\em Communications in Mathematical Physics}, 304(2):411--432, 2011,
  \href{http://arxiv.org/abs/0910.5534}{arXiv:0910.5534}.

\bibitem[ST01]{Seidel:2000}
Paul Seidel and Richard~P Thomas.
\newblock {Braid group actions on derived categories of coherent sheaves}.
\newblock {\em Duke Mathematical Journal}, 108:37--108, 2001,
  \href{http://arxiv.org/abs/math/0001043}{arXiv:math/0001043}.

\bibitem[Tho10]{Thomas:2010tg}
Richard~P Thomas.
\newblock {An exercise in mirror symmetry},
\newblock {\em Proceedings of the
  International Congress of Mathematicians: Hyderabad, India, 2010},  \url{http://www2.imperial.ac.uk/~rpwt/ICM.PDF}.


%
\bibitem[Wey03]{Weyman:2007wu}
Jerzy Weyman.
\newblock {\em {Cohomology of Vector Bundles and Syzygies}}, volume 149 of {\em
  Cambridge Tracts in Mathematics}.
\newblock Cambridge, 2003.

\end{thebibliography}

\opt{lms}{
\affiliationone{Will Donovan\\
\affiliationAddress.\\
\affiliationCountry\\
\email{\affiliationEmail}}
}

\end{document}